\documentclass[leqno,12pt]{amsart}
\usepackage{amssymb}
\usepackage{amsmath}
\usepackage{enumerate}
\usepackage{amsfonts}
\usepackage{hyperref}
\usepackage{mathrsfs}

\usepackage[headheight=18pt, top=20mm, bottom=20mm, left=22mm, right=22mm]{geometry}

\usepackage{tikz}
\usepackage{pgfplots}
\pgfplotsset{compat=1.18}

\definecolor{darkgreen}{RGB}{45, 119, 75}

\newtheorem{theorem}{Theorem}[section]
\newtheorem{corollary}[theorem]{Corollary}
\newtheorem{lemma}[theorem]{Lemma}
\newtheorem{proposition}[theorem]{Proposition}

\numberwithin{equation}{section}

\hypersetup{
    colorlinks=true,
    linkcolor= blue,
    citecolor =cyan,
    urlcolor = teal,
}

\makeatletter
\@ifundefined{c@part}{\newcounter{part}}{}

\renewcommand\part{%
  \@ifstar{\@spart}{\@part}%
}

\def\@part#1{%
  \refstepcounter{part}%
  \addcontentsline{toc}{part}{\protect\numberline{\thepart}#1}%
  \par\vspace{1\bigskipamount}
  \begin{center}%
    \normalfont\large\bfseries Part\ \thepart\quad #1%
  \end{center}%
  \vspace{1\bigskipamount}%
  \@afterheading
}

\def\@spart#1{%
  \addcontentsline{toc}{part}{#1}%
  \par\vspace{1\bigskipamount}
  \begin{center}%
    \normalfont\large\bfseries #1%
  \end{center}%
  \vspace{1\bigskipamount}%
  \@afterheading
}
\makeatother
  
\begin{document}

\title[Poisson representations of Dunkl harmonic functions]
{Poisson representations of Dunkl harmonic functions satisfying a Carleson condition}


\author[Jacek Dziubański]{Jacek Dziubański}
\author[Rafał Łyżwa]{Rafał Łyżwa}

\begin{abstract}
We prove that a smooth Dunkl harmonic function in the
upper half-space whose Dunkl gradient satisfies a
Carleson measure condition is the Poisson integral of
a function in $\mathrm{BMO}(dw)$, defined using Euclidean
balls. The result holds for arbitrary finite root systems
and nonnegative multiplicities. The proof uses an
estimate for increments from the origin, obtained from
a finite-dimensional system indexed by the reflection
group, together with the atomic characterization of
the Dunkl Hardy space.
\end{abstract}

\address{Jacek Dziubański, Uniwersytet Wroc\l awski,
Instytut Matematyczny,
Pl. Grunwaldzki 2,
50-384 Wroc\l aw,
Poland}
\email{jdziuban@math.uni.wroc.pl}

\address{Rafał Łyżwa, Uniwersytet Wroc\l awski,
Instytut Matematyczny,
Pl. Grunwaldzki 2,
50-384 Wroc\l aw,
Poland}
\email{rafal.lyzwa@math.uni.wroc.pl}

\maketitle

\section{Introduction}

The relation between functions of bounded mean oscillation
and Carleson measures associated with their Poisson
extensions is classical. Fefferman and Stein \cite{FS}
proved that a locally integrable function $f$ satisfying
\[
    \int_{\mathbb R^N}
    \frac{|f(\mathbf x)|}{(1+\|\mathbf x\|)^{N+1}}
    \,d\mathbf x<\infty
\]
belongs to $\mathrm{BMO}(\mathbb R^N)$ if and only if
\[
    t\,|\nabla_{t,\mathbf x}u(t,\mathbf x)|^2
    \,d\mathbf x\,dt
\]
is a Carleson measure, where $u$ is the Poisson extension
of $f$. Subsequently, Fabes, Johnson, and Neri
\cite{Fabes} proved that every harmonic function on
the upper half-space satisfying this Carleson condition
is the Poisson extension of a function in
$\mathrm{BMO}(\mathbb R^N)$.
The purpose of the present paper is to establish
an analogue of the latter result for the Dunkl Laplacian.

We first introduce the notation needed to state
our result. Let $R\subset\mathbb R^N\setminus\{0\}$
be a finite reduced root system, normalized by
$\|\alpha\|^2=2$ for $\alpha\in R$.
For $\alpha\in R$, let
\[
    \sigma_\alpha(\mathbf x)
    =\mathbf x-\langle\alpha,\mathbf x\rangle\alpha
\]
be the corresponding reflection, and let $G$ be
the finite group generated by these reflections.
We fix a $G$-invariant multiplicity function
$k:R\to[0,\infty)$ and a positive subsystem $R^+$.
The associated measure is
\[
    dw(\mathbf x)=w(\mathbf x)\,d\mathbf x,
    \qquad
    w(\mathbf x)
    =\prod_{\alpha\in R^+}
        |\langle\alpha,\mathbf x\rangle|^{2k(\alpha)}.
\]
We write
\[
    \mathbf N=N+2\sum_{\alpha\in R^+}k(\alpha)
\]
for the homogeneous dimension.

For $\xi\in\mathbb R^N$, the Dunkl operator
introduced in \cite{Dunkl} is defined by
\[
    T_\xi f(\mathbf x)
    =
    \partial_\xi f(\mathbf x)
    +\sum_{\alpha\in R^+}k(\alpha)
        \langle\alpha,\xi\rangle
        \frac{f(\mathbf x)-f(\sigma_\alpha(\mathbf x))}
             {\langle\alpha,\mathbf x\rangle}.
\]
For $f\in C^1(\mathbb R^N)$, the quotients in this
formula are understood by continuous extension
across the reflection hyperplanes.
We set $T_j=T_{e_j}$, where $e_1,\ldots,e_N$ is
the canonical basis of $\mathbb R^N$, and denote
the Dunkl Laplacian by
\[
    {\Delta_k}=\sum_{j=1}^N T_j^2.
\]
Its self-adjoint realization on $L^2(dw)$ generates
the heat semigroup $e^{t{\Delta_k}}$.
The associated Poisson semigroup is
\[
    P_t=e^{-t\sqrt{-{\Delta_k}}},
    \qquad
    P_tf(\mathbf x)
    =\int_{\mathbb R^N}
        p_t(\mathbf x,\mathbf y)f(\mathbf y)
        \,dw(\mathbf y).
\]

We consider one-dimensional extension of $\mathbb{R}^N$, namely
\[
    \mathbb R^{1+N}=\mathbb{R} \times \mathbb{R}^N
\]
equipped with a root system $\mathcal{R}=\{(0,\alpha)\::\: \alpha \in R\}$ and multiplicity function ${k}((0,\alpha))=k(\alpha)$. We denote by  $\mathcal{G}$ the corresponding reflection group on $\mathbb{R}^{1+N}$, which is isomorphic to $G$.. Then $d\mu((t,\mathbf{x}))=dt\,dw(\mathbf{x})$ is the corresponding Dunkl measure. With a little abuse of notation, we denote by the same letter $T_j$, $j=1,\ldots, N$, the corresponding Dunkl operators i $\mathbb{R}^{1+N}$. Set $T_0=\frac{\partial}{\partial t}$. Then  the operator
\[
    \mathcal L=T_0^2+{\Delta_{k,\mathbf{x}}}.
\]
is the Dunkl Laplacian in $\mathbb{R}^{1+N}$.
For a smooth function $u$ defined on $\mathcal{G}$-invariant open subset $\Omega$ in $\mathbb R_+^{1+N}$, we
write
\[
    \widetilde\nabla_k u
    =(T_0u,T_1u,\ldots,T_Nu).
\]
We say that $u$ is $\mathcal L$-harmonic in $\Omega$ if
$\mathcal Lu(t,\mathbf{x})=0$ for $(t,\mathbf{x}) \in \Omega$.

All balls below are Euclidean balls. For
$f\in L^1_{\mathrm{loc}}(dw)$, let
\[
    f_B=\frac1{w(B)}\int_Bf\,dw,
    \qquad
    \|f\|_{\mathrm{BMO}(dw)}
    =
    \sup_B\frac1{w(B)}\int_B|f-f_B|\,dw.
\]
The space $\mathrm{BMO}(dw)$ consists of the functions
for which this quantity is finite, with functions
differing by a constant identified.

A nonnegative Borel measure $\nu$ on
$\mathbb R_+^{1+N}$ is called a Carleson measure if
\[
    \|\nu\|_{\mathcal C}
    =
    \sup_{\mathbf x_0\in\mathbb R^N,\ r>0}
    \frac{\nu((0,r)\times B(\mathbf x_0,r))}
         {w(B(\mathbf x_0,r))}
    <\infty.
\]
Correspondingly, we put
\[
    \|u\|_{**}
    =
    \sup_{\mathbf x_0\in\mathbb R^N,\ r>0}
    \left(
        \frac1{w(B(\mathbf x_0,r))}
        \int_0^r\int_{B(\mathbf x_0,r)}
        t\,|\widetilde\nabla_k u(t,\mathbf x)|^2
        \,dw(\mathbf x)\,dt
    \right)^{1/2}.
\]
Our main result is the following theorem.

\begin{theorem}\label{thm:main}
Let $R$ be a finite reduced root system and let
$k\geq0$ be a multiplicity function.
Suppose that
$u\in C^\infty(\mathbb R_+^{1+N};\mathbb C)$ satisfies
\[
    \mathcal Lu=0,
    \qquad
    \|u\|_{**}<\infty.
\]
Then there exists $f\in\mathrm{BMO}(dw)$ such that
\begin{equation}\label{eq:main-representation}
    u(t,\mathbf x)
    =\int_{\mathbb R^N}
        p_t(\mathbf x,\mathbf y)f(\mathbf y)
        \,dw(\mathbf y),
    \qquad
    (t,\mathbf x)\in\mathbb R_+^{1+N}.
\end{equation}
The integral in \eqref{eq:main-representation}
is absolutely convergent, and
\[
    \|f\|_{\mathrm{BMO}(dw)}
    \leq C\|u\|_{**},
\]
where $C$ depends only on $R$, $k$, and $N$.
\end{theorem}

When $k\equiv0$, the operators $T_j$ are the ordinary
partial derivatives and $dw$ is Lebesgue measure.
In this case Theorem~\ref{thm:main} reduces to the
classical result of Fabes, Johnson, and Neri.
The one-dimensional Dunkl case was established
in \cite{Lyzwa}. Here we consider arbitrary
finite root systems.

The converse implication was proved by Dziuba\'nski
and Hejna \cite[Corollary~4.17]{DH-Uchiyama}:
for every $f\in\mathrm{BMO}(dw)$, its Poisson extension
$u(t,\mathbf x)=P_tf(\mathbf x)$ satisfies
\[
    \|u\|_{**}\leq C\|f\|_{\mathrm{BMO}(dw)}.
\]
Together with Theorem~\ref{thm:main}, this characterizes
the Poisson extensions of $\mathrm{BMO}(dw)$ functions
among smooth $\mathcal L$-harmonic functions
by the Carleson condition.

The proof uses the theory of harmonic functions
and the Hardy space $H^1_{\Delta_k}$ developed in
\cite{ADH-JFAA}, together with its Coifman--Weiss
atomic characterization obtained in \cite{DH-Studia}.
We also use the Dunkl mean-value formula
of Gallardo and Rejeb \cite{GR}.
An estimate needed in the argument is
\[
    |f(\mathbf x)-f(0)|
    \leq \|\mathbf x\|
    \sup_{\|\mathbf y\|\leq\|\mathbf x\|}
        |\nabla_kf(\mathbf y)|,
    \qquad f\in C^1(\mathbb R^N),
\]
where $\nabla_kf=(T_1f,\ldots,T_Nf)$.
It follows by considering the restrictions of $f$
to the rays through the points of a $G$-orbit.
Applied to the horizontal sections of $u$, this
estimate gives the logarithmic growth needed
to prove their Poisson representation. In the one-dimensional case, we also obtain a sharp estimate
for the ordinary derivative in terms of the Dunkl derivative:
\[
    \|f'\|_{L^\infty([-a,a])}
    \leq \frac{1+4k}{1+2k}
           \|Tf\|_{L^\infty([-a,a])}.
\]
The proof gives an explicit inverse of the operator relating
$f'$ to $Tf$.
The theorem then follows from a Carleson measure
criterion for $\mathrm{BMO}(dw)$ and a weak-$*$
compactness argument.

The paper is organized as follows. Section~\ref{sec:pre} collects the preliminaries. In Section~\ref{sec:Dunkl-derivative-estimates} we establish estimates in terms of Dunkl derivatives. Section~\ref{sec:pointwise} contains pointwise estimates for harmonic functions and the Poisson representation of their vertical translates. Section~\ref{sec:carleson} proves a Carleson measure criterion for $\mathrm{BMO}(dw)$. The proof of Theorem~\ref{thm:main} is given in Section~\ref{sec:proof-main-theorem}.

\section{Preliminaries}\label{sec:pre}

We retain the notation introduced in Section~1.
Implicit constants may depend on $R$, $k$, and $N$,
but not on the functions or scale parameters under
consideration. We write $A\lesssim B$ if $A\leq CB$
for such a constant, and $A\asymp B$ if both
$A\lesssim B$ and $B\lesssim A$ hold.

\subsection{The measure and Dunkl operators}

For a measurable set $E$, we write $w(E)=\int_Edw$.
The measure $dw$ is $G$-invariant and satisfies
\begin{equation}\label{eq:volume-estimate}
    w(B(\mathbf x,r))
    \asymp
    r^N\prod_{\alpha\in R^+}
    \bigl(|\langle\alpha,\mathbf x\rangle|+r
    \bigr)^{2k(\alpha)}.
\end{equation}
In particular, it is doubling. More precisely,
for $0<r\leq s$,
\begin{equation}\label{eq:volume-growth}
    C^{-1}\left(\frac sr\right)^N
    \leq
    \frac{w(B(\mathbf x,s))}{w(B(\mathbf x,r))}
    \leq
    C\left(\frac sr\right)^{\mathbf N}.
\end{equation}
Moreover,
\[
    w(B(a\mathbf x,ar))
    =a^{\mathbf N}w(B(\mathbf x,r)),
    \qquad a>0.
\]

We use the orbit distance and orbit notation
\[
    d(\mathbf x,\mathbf y)
    =\min_{g\in G}\|\mathbf x-g(\mathbf y)\|,
    \qquad
    \mathcal O(B)=\bigcup_{g\in G}gB.
\]
Thus
\[
    \{\mathbf y:d(\mathbf x,\mathbf y)<r\}
    =\mathcal O(B(\mathbf x,r)),
    \qquad
    w(\mathcal O(B))\leq |G|w(B).
\]

The Dunkl operators $T_j$ commute and are
skew-symmetric with respect to $dw$:
\begin{equation}\label{eq:Dunkl-integration-by-parts}
    \int_{\mathbb R^N}(T_jf)g\,dw
    =-\int_{\mathbb R^N}f(T_jg)\,dw,
\end{equation}
whenever, for example, $f\in C_c^\infty(\mathbb R^N)$
and $g\in C^1(\mathbb R^N)$.
If $\chi$ is $G$-invariant, then
\begin{equation}\label{eq:Dunkl-Leibniz-invariant}
    T_j(\chi f)=\chi T_jf+(\partial_j\chi)f.
\end{equation}
We refer to \cite{Dunkl,Roesler3,Anker} for the
basic facts concerning Dunkl operators.

\subsection{The Poisson semigroup}

The Dunkl Laplacian is essentially self-adjoint
on $\mathcal S(\mathbb R^N)$ in $L^2(dw)$.
Writing $H_t=e^{t{\Delta_k}}$, the subordination formula
takes the form
\[
    P_t
    =
    \frac1{\sqrt\pi}\int_0^\infty
        e^{-s}H_{t^2/(4s)}\,\frac{ds}{\sqrt s}.
\]
The Poisson kernel $p_t$ is smooth, positive,
and symmetric, and
\begin{equation}\label{eq:Poisson-mass}
    \int_{\mathbb R^N}p_t(\mathbf x,\mathbf y)
        \,dw(\mathbf y)=1.
\end{equation}
It satisfies the semigroup identity and
\begin{equation}\label{eq:Poisson-Dunkl-symmetry}
    T_{j,\mathbf x}p_t(\mathbf x,\mathbf y)
    =-T_{j,\mathbf y}p_t(\mathbf x,\mathbf y).
\end{equation}

We shall use the following kernel estimates
from \cite[Proposition 5.1]{ADH-JFAA}:
\begin{equation}\label{eq:Poisson-size}
    p_t(\mathbf x,\mathbf y)
    \lesssim
    \frac1{w(B(\mathbf x,t+d(\mathbf x,\mathbf y)))}
    \frac{t}{t+d(\mathbf x,\mathbf y)},
\end{equation}
\begin{equation}\label{eq:Poisson-Dunkl-derivative}
    |T_{j,\mathbf x}p_t(\mathbf x,\mathbf y)|
    \lesssim
    \frac1{w(B(\mathbf x,t+d(\mathbf x,\mathbf y)))}
    \frac1{t+d(\mathbf x,\mathbf y)},
\end{equation}
and, for $m\in\mathbb N$ and multi-indices $\beta,\eta$,
\begin{equation}\label{eq:Poisson-ordinary-derivatives}
    |\partial_t^m\partial_{\mathbf x}^{\beta}
        \partial_{\mathbf y}^{\eta}
        p_t(\mathbf x,\mathbf y)|
    \leq C_{m,\beta,\eta}
        t^{-m-|\beta|-|\eta|}
        p_t(\mathbf x,\mathbf y).
\end{equation}
Here and below, $\mathbb N$ includes zero.
At the origin,
\begin{equation}\label{eq:Poisson-at-origin}
    p_t(0,\mathbf y)
    =c_{R,k}\frac{t}
        {(t^2+\|\mathbf y\|^2)^{(\mathbf N+1)/2}},
\end{equation}
where $c_{R,k}>0$ is determined by
\eqref{eq:Poisson-mass}.

We set
\begin{equation}\label{eq:definition-Q}
    Q_t=t\sqrt{-{\Delta_k}}\,e^{-t\sqrt{-{\Delta_k}}}
        =-t\partial_tP_t,
    \qquad
    q_t(\mathbf x,\mathbf y)
        =-t\partial_tp_t(\mathbf x,\mathbf y).
\end{equation}
In particular,
$|q_t(\mathbf x,\mathbf y)|\lesssim
p_t(\mathbf x,\mathbf y)$.

\subsection{Harmonic functions}

We recall the volume mean-value formula of Gallardo
and Rejeb. Let $h_k(r,\mathbf x,\mathbf y)$ denote
the harmonic kernel introduced in \cite[Section~3]{GR}.
It is nonnegative and satisfies
\[
    \operatorname{supp}h_k(r,\mathbf x,\cdot)
    \subseteq
    \bigcup_{g\in G}\overline{B(g(\mathbf x),r)},
    \qquad
    \int_{\mathbb R^N}h_k(r,\mathbf x,\mathbf y)
        \,dw(\mathbf y)=w(B(0,r)).
\]
For an open $G$-invariant set $\Omega$, a continuous
function $f$ on $\Omega$, and
$\overline{B(\mathbf x,r)}\subset\Omega$, define
\[
    M_B^r f(\mathbf x)
    =
    \frac1{w(B(0,r))}
    \int_\Omega h_k(r,\mathbf x,\mathbf y)
        f(\mathbf y)\,dw(\mathbf y).
\]
We also set $M_B^0f(\mathbf x)=f(\mathbf x)$.

\begin{theorem}[{\cite[Theorem~3.1]{GR}}]
\label{teo:GR}
Let $\Omega\subseteq\mathbb R^N$ be open and
$G$-invariant, and let $f\in C^2(\Omega)$.
If $\mathbf x\in\Omega$ and $\rho>0$ satisfy
$\overline{B(\mathbf x,\rho)}\subset\Omega$, then,
for every $0<a\leq\rho/3$,
\begin{equation}\label{eq:GR-mean-value}
    M_B^a f(\mathbf x)
    =
    f(\mathbf x)
    +
    \frac1{a^{\mathbf N}}
    \int_0^a r^{\mathbf N-1}
        \int_0^r M_B^s({\Delta_k} f)(\mathbf x)
            \,s\,ds\,dr.
\end{equation}
In particular, if ${\Delta_k} f=0$ in $\Omega$, then
\[
    M_B^a f(\mathbf x)=f(\mathbf x).
\]
\end{theorem}

We have the following theorem.

\begin{theorem}[{\cite[Theorem~7.3]{ADH-JFAA}}]
\label{thm:bounded-Poisson-representation}
Every bounded $\mathcal L$-harmonic function
$v\in C^2(\mathbb R_+^{1+N};\mathbb C)$
has a representation
\[
    v(t,\mathbf x)=P_th(\mathbf x)
\]
for some $h\in L^\infty(dw)$.
If, in addition, $v$ extends continuously to
$[0,\infty)\times\mathbb R^N$, then
\begin{equation}\label{eq:bounded-harmonic-representation}
    v(t,\mathbf x)=P_t(v(0,\cdot))(\mathbf x).
\end{equation}
\end{theorem}

The second assertion follows by identifying $h$
with the boundary values of $v$, using
\cite[Corollary~5.4]{ADH-JFAA}.
The complex-valued case follows by applying
the result to the real and imaginary parts.

\subsection{Dunkl translations}

We shall use Dunkl translations of radial functions.
Let $\mu_{\mathbf x}$ be the probability measure
representing the Dunkl intertwining operator,
supported in $\operatorname{conv}\mathcal{O}(\mathbf x)$.
R\"osler's formula
\cite[Theorem~5.1 and its proof]{Roesler-product}
extends the translation of radial functions to
bounded Borel radial functions
$f(\mathbf y)=\widetilde f(\|\mathbf y\|)$ by
\[
    \tau_{\mathbf x}f(-\mathbf y)
    =
    \int_{\mathbb R^N}
    \widetilde f\!\left(
        \sqrt{\|\mathbf x\|^2+\|\mathbf y\|^2
              -2\langle\eta,\mathbf y\rangle}
    \right)\,d\mu_{\mathbf x}(\eta).
\]
In particular, this formula defines the translated
characteristic function of a ball and shows that
Dunkl translations preserve positivity on radial
functions. It will be used in
$\mathbb R^{1+N}$ for the root system
$\{(0,\alpha):\alpha\in R\}$.
\subsection{Hardy space and duality}

For $g\in L^1(dw)$, let
\[
    \mathcal M_Pg(\mathbf x)
    =
    \sup_{{t>0\;
        \|\mathbf x-\mathbf y\|<t}}
        |P_tg(\mathbf y)|.
\]
The Hardy space associated with ${\Delta_k}$ is
\[
    H^1_{\Delta_k}
    =
    \{g\in L^1(dw):\mathcal M_Pg\in L^1(dw)\},
    \qquad
    \|g\|_{H^1_{\Delta_k}}
    =\|\mathcal M_Pg\|_{L^1(dw)}.
\]
Its square-function characterization
\cite{ADH-JFAA} gives
\begin{equation}\label{eq:H1-square-characterization}
    \|g\|_{H^1_{\Delta_k}}\asymp\|Sg\|_{L^1(dw)},
\end{equation}
where
\[
    Sg(\mathbf x)
    =
    \left(
        \int_0^\infty
        \int_{\|\mathbf x-\mathbf y\|<t}
        |Q_tg(\mathbf y)|^2
        \frac{dw(\mathbf y)\,dt}
             {t\,w(B(\mathbf x,t))}
    \right)^{1/2}.
\]

We also recall the atomic Hardy space on a space of homogeneous type $(\mathbb{R}^N,\|\cdot-\cdot\|,dw)$.
A $(1,\infty)$-atom is a measurable function $a$
supported in a Euclidean ball $B$ such that
\[
    \|a\|_{L^\infty(dw)}\leq w(B)^{-1},
    \qquad
    \int_{\mathbb R^N}a\,dw=0.
\]
The space $H^1_{\mathrm{at}}(dw)$ consists of all
$f\in L^1(dw)$ admitting a representation
\[
    f=\sum_{j=1}^\infty\lambda_j a_j
    \quad\text{in }L^1(dw),
    \qquad
    \sum_{j=1}^\infty|\lambda_j|<\infty,
\]
where the $a_j$ are $(1,\infty)$-atoms.
Its norm is the infimum of $\sum_j|\lambda_j|$
over all such representations.

By \cite{DH-Studia}, $H^1_{\Delta_k}$ coincides, with $H^1_{\mathrm{at}}(dw)$.

In particular, bounded, compactly supported functions
with integral zero form a dense subspace of
$H^1_{\Delta_k}$.

The atomic characterization and \cite{CW} imply
\begin{equation}\label{eq:H1-BMO-duality}
    (H^1_{\Delta_k})^*=\mathrm{BMO}(dw)
\end{equation}
with equivalent norms.
The functional associated with $f\in\mathrm{BMO}(dw)$
is denoted by $\Phi_f$ and is initially defined by
\[
    \langle\Phi_f,g\rangle
    =\int_{\mathbb R^N}fg\,dw
\]
for bounded, compactly supported $g$ of integral zero.

We shall also use
\begin{equation}\label{eq:absolutely-integrable-pairing}
    \langle\Phi_f,g\rangle=\int_{\mathbb R^N}fg\,dw
    \quad\text{if }f\in\mathrm{BMO}(dw),\
    g\in H^1_{\Delta_k},\ \text{and }fg\in L^1(dw).
\end{equation}
This follows by the standard truncation argument;
see \cite[Chapter~IV]{Stein-Harmonic} for the classical
$H^1$--$\mathrm{BMO}$ theory.
Indeed, for real-valued $f$, its truncations
$f_m=\max(-m,\min(m,f))$ have uniformly bounded
BMO seminorms and converge to $f$ locally in $L^1(dw)$.
The corresponding functionals converge on the dense
subspace above, and hence on all of $H^1_{\Delta_k}$.
Dominated convergence applied to $f_mg$ proves
\eqref{eq:absolutely-integrable-pairing}.
The complex-valued case follows by splitting $f$
into its real and imaginary parts.

The space $H^1_{\Delta_k}$ is separable: by the Riesz
transform characterization
\cite[Theorem~2.5]{ADH-JFAA}, the map
$h\mapsto(h,R_1h,\ldots,R_Nh)$, where
$R_j=T_j(-{\Delta_k})^{-1/2}$, identifies it, with an
equivalent norm, with a subspace of the separable
space $(L^1(dw))^{N+1}$.

\subsection{Integrability of Poisson extensions}

Put
\[
    \rho(\mathbf x)=p_1(0,\mathbf x)
    \asymp(1+\|\mathbf x\|)^{-\mathbf N-1}.
\]
Every function of at most logarithmic growth belongs
to $L^1(\rho\,dw)$.
The same is true for every $f\in\mathrm{BMO}(dw)$;
more precisely,
\begin{equation}\label{eq:BMO-weighted-integrability}
    \int_{\mathbb R^N}|f(\mathbf x)|\rho(\mathbf x)
        \,dw(\mathbf x)
    \lesssim
    \|f\|_{\mathrm{BMO}(dw)}
    +|f_{B(0,1)}|.
\end{equation}
To see this, decompose $\mathbb R^N$ into $B(0,1)$
and concentric dyadic annuli, and use
\[
    |f_{B(0,2^j)}-f_{B(0,1)}|
    \lesssim j\|f\|_{\mathrm{BMO}(dw)}.
\]
The resulting series is bounded by a constant
multiple of $\sum_{j\geq0}(j+1)2^{-j}$.

If $f\in L^1(\rho\,dw)$, the estimates
\eqref{eq:Poisson-size}--\eqref{eq:Poisson-ordinary-derivatives}
imply that $P_tf$ is absolutely convergent and defines
a smooth $\mathcal L$-harmonic function. Moreover,
\begin{equation}\label{eq:Poisson-differentiation}
\begin{aligned}
    \partial_tP_tf(\mathbf x)
    &=\int_{\mathbb R^N}
        \partial_tp_t(\mathbf x,\mathbf y)f(\mathbf y)
        \,dw(\mathbf y),\\
    T_jP_tf(\mathbf x)
    &=\int_{\mathbb R^N}
        T_{j,\mathbf x}p_t(\mathbf x,\mathbf y)
        f(\mathbf y)\,dw(\mathbf y).
\end{aligned}
\end{equation}
These assertions follow by dominated convergence,
using locally uniform majorants for the kernels
and their derivatives.
If $f$ is continuous at zero, then
\begin{equation}\label{eq:Poisson-limit-at-origin}
    \lim_{t\to0^+}P_tf(0)=f(0).
\end{equation}
This follows from \eqref{eq:Poisson-at-origin} by
splitting the integral into a small ball centered
at zero and its complement.

\subsection{Carleson measures and tents}

For a Euclidean ball $B$, set
\[
    \widehat B
    =
    \{(t,\mathbf y):t>0,\ B(\mathbf y,t)\subset B\}.
\]
Since $\widehat B\subset Q(2B)$ and
$Q(B)\subset\widehat{2B}$, where
$Q(B)=(0,r_B)\times B$, the doubling property
$w(2B)\leq Cw(B)$ of the Dunkl measure $dw$ yields
\[
    \|\nu\|_{\mathcal C}
    \asymp \sup_B\frac{\nu(\widehat B)}{w(B)}.
\]

\section{Estimates in terms of Dunkl derivatives}
\label{sec:Dunkl-derivative-estimates}

We first prove an estimate for the increment of a function
from the origin in terms of its Dunkl gradient.
Lemma~\ref{lem:increment-at-origin} will be used in the
proof of the main theorem.

We call a vector $\xi\in\mathbb R^N$ regular if
$\langle\alpha,\xi\rangle\neq0$ for every
$\alpha\in R$, that is, if it lies outside all reflection
hyperplanes. In particular, regular unit vectors form
a dense subset of the unit sphere.

\begin{lemma}\label{lem:increment-at-origin}
For every $f\in C^1(\mathbb R^N;\mathbb C)$ and
$\mathbf x\in\mathbb R^N\setminus\{0\}$,
\begin{equation}\label{eq:increment-at-origin}
\begin{aligned}
    |f(\mathbf x)-f(0)|
    &\leq \|\mathbf x\|
        \max_{g\in G}\sup_{0\leq t\leq1}
        \left|
            T_{g(\mathbf x/\|\mathbf x\|)}
            f(tg(\mathbf x))
        \right|\leq \|\mathbf x\|
        \sup_{\|\mathbf y\|\leq\|\mathbf x\|}
        |\nabla_k f(\mathbf y)|,
\end{aligned}
\end{equation}
where
\[
    \nabla_k f=(T_1f,\ldots,T_Nf),
    \qquad
    T_{\boldsymbol\xi}f
    =\sum_{j=1}^N\xi_jT_jf.
\]
In particular, if $T_jf=0$ for every $j=1,\ldots,N$,
then $f$ is constant.
\end{lemma}

\begin{proof}
For $\alpha\in R^+$, define the permutation operator
$S_\alpha$ on $\mathbb C^{G}$ by
\[
    (S_\alpha v)_g=v_{\sigma_\alpha g},
    \qquad g\in G,
\]
and put
\[
    A=\sum_{\alpha\in R^+}k(\alpha)(I-S_\alpha),
    \qquad
    \gamma=\sum_{\alpha\in R^+}k(\alpha).
\]
Writing $A=\gamma I-K$, where
$K=\sum_{\alpha\in R^+}k(\alpha)S_\alpha$, we see that
\[
    e^{-sA}=e^{-\gamma s}e^{sK},
    \qquad s\geq0,
\]
has nonnegative entries. Moreover, $A\mathbf1=0$, where $\mathbf1=(1)_{g\in G}$,
and hence $e^{-sA}\mathbf1=\mathbf1$.
Thus the entries of $e^{-sA}$ are nonnegative and each
row sums to one. Consequently, for every $g\in G$,
\[
    |(e^{-sA}v)_g|
    \leq\sum_{h\in G}(e^{-sA})_{g,h}|v_h|
    \leq\|v\|_{\ell^\infty(G)},
\]
which proves
\begin{equation}\label{eq:orbit-semigroup-contraction}
    \|e^{-sA}v\|_{\ell^\infty(G)}
    \leq\|v\|_{\ell^\infty(G)},
    \qquad s\geq0.
\end{equation}

Fix a unit vector $\xi$ such that
$\langle\alpha,\xi\rangle\neq0$ for every $\alpha\in R$.
For $r\geq0$ and $g\in G$, set
\[
    F_g(r)=f(rg(\xi)),
    \qquad
    S_g(r)=T_{g(\xi)}f(rg(\xi)).
\]
The definition of the Dunkl operators gives, for $r>0$,
\[
    S_g(r)
    =F'_g(r)
     +\frac1r\sum_{\alpha\in R^+}k(\alpha)
        \bigl(F_g(r)-F_{\sigma_\alpha \circ g}(r)\bigr).
\]
Thus
\[
    F'(r)+\frac1rAF(r)=S(r).
\]
Since $F(0)=f(0)\mathbf1$ and $A\mathbf1=0$, the vector
$U(r)=F(r)-f(0)\mathbf1$ satisfies
\[
    U'(r)+\frac1rAU(r)=S(r),
    \qquad U(0)=0.
\]
For $r>0$, put $r^A=e^{(\log r)A}$.
Since $U'(r)+r^{-1}AU(r)=S(r)$, we have
\[
    \frac{d}{dr}\bigl(r^AU(r)\bigr)
    =r^A\bigl(U'(r)+r^{-1}AU(r)\bigr)
    =r^AS(r).
\]
Integrating over $[\varepsilon,r]$ and multiplying
by $r^{-A}$ gives
\[
    U(r)
    =e^{-\log(r/\varepsilon)A}U(\varepsilon)
     +\int_\varepsilon^r
        e^{-\log(r/s)A}S(s)\,ds.
\]
By \eqref{eq:orbit-semigroup-contraction} and continuity
at zero, we may let $\varepsilon\to0^+$ to obtain
\[
    U(r)=\int_0^r e^{-\log(r/s)A}S(s)\,ds.
\]
Therefore,
\[
\begin{aligned}
    |f(r\xi)-f(0)|
    &\leq \|U(r)\|_{\ell^\infty(G)}\leq \int_0^r
        \max_{g\in G}|T_{g(\xi)}f(sg(\xi))|\,ds\\
    &=r\int_0^1
        \max_{g\in G}|T_{g(\xi)}f(trg(\xi))|\,dt\leq r\max_{g\in G}\sup_{0\leq t\leq1}
        |T_{g(\xi)}f(trg(\xi))|.
\end{aligned}
\]
Since $f\in C^1$, its Dunkl derivatives are continuous.
The finiteness of $G$ and compactness of $[0,1]$
therefore show that the last expression is continuous
in $\xi$. By density of regular unit vectors, the
estimate holds for every unit vector $\xi$.
Taking $r=\|\mathbf x\|$ and
$\xi=\mathbf x/\|\mathbf x\|$ proves the first inequality
in \eqref{eq:increment-at-origin}.
The second follows from
\[
    |T_{g(\xi)}f(\mathbf y)|
    =\left|\sum_{j=1}^N(g(\xi))_jT_jf(\mathbf y)\right|
    \leq|\nabla_k f(\mathbf y)|,
\]
since $\|g(\xi)\|=1$.
The final assertion follows immediately.
\end{proof}

\begin{corollary}\label{cor:radial-derivative}
Let $\gamma=\sum_{\alpha\in R}\frac{k(\alpha)}{2}$.
For every $f\in C^1(\mathbb R^N;\mathbb C)$,
every unit vector $v\in\mathbb R^N$, and every $r>0$,
\begin{equation}\label{eq:radial-derivative}
\begin{aligned}
    |\partial_v f(rv)|
    &\leq (1+2\gamma)
        \max_{g\in G}\sup_{0\leq t\leq1}
            |T_{g(v)}f(trg(v))|\leq (1+2\gamma)
        \sup_{\mathbf y\in B(0,r)}
            |\nabla_k f(\mathbf y)|.
\end{aligned}
\end{equation}
\end{corollary}

\begin{proof}
First assume that $v$ is regular, and put
\[
    M=\max_{g\in G}\sup_{0\leq t\leq1}
        |T_{g(v)}f(trg(v))|.
\]
Applying Lemma~\ref{lem:increment-at-origin}
at $rg(v)$ and using the group property, we obtain
\[
    |f(rg(v))-f(0)|\leq rM,
    \qquad g\in G.
\]
Hence, for every $\alpha\in R$,
\[
    |f(rv)-f(r\sigma_\alpha(v))|\leq2rM.
\]
By the definition of $T_v$,
\[
    \partial_v f(rv)
    =T_vf(rv)
     -\frac1r\sum_{\alpha\in R^+}k(\alpha)
        \bigl(f(rv)-f(r\sigma_\alpha(v))\bigr).
\]
Consequently,
\[
    |\partial_v f(rv)|
    \leq M+2\gamma M=(1+2\gamma)M.
\]
The second inequality in \eqref{eq:radial-derivative}
follows from
$|T_{g(v)}f(\mathbf y)|\leq|\nabla_k f(\mathbf y)|$
and continuity, which allows the closed ball
to be replaced by $B(0,r)$.
Finally, continuity of the Dunkl derivatives,
finiteness of $G$, and compactness of $[0,1]$
extend the estimate to every unit vector $v$
by density of regular unit vectors.
\end{proof}

In dimension one, the estimate of Corollary~\ref{cor:radial-derivative}
can be sharpened to the optimal constant.
We also obtain an explicit inverse of the operator
relating the ordinary and Dunkl derivatives.

\begin{proposition}\label{prop:rank-one}
Let $k\geq0$ and $a>0$. For every
$f\in C^1([-a,a];\mathbb C)$,
\begin{equation}\label{eq:rank-one-estimate}
    \|f'\|_{L^\infty([-a,a])}
    \leq\frac{1+4k}{1+2k}
        \|Tf\|_{L^\infty([-a,a])},
\end{equation}
where
\[
    Tf(x)=f'(x)+k\frac{f(x)-f(-x)}{x}
\]
is understood by continuity at $x=0$.
The constant in \eqref{eq:rank-one-estimate} is optimal.

More precisely, the operator
\[
    U_kh(x)=h(x)+k\int_{-1}^{1}h(sx)\,ds
\]
is invertible on $C([-a,a];\mathbb C)$, and
\begin{equation}\label{eq:rank-one-inverse}
    U_k^{-1}h(x)
    =h(x)-k\int_{-1}^{1}|s|^{2k}h(sx)\,ds.
\end{equation}
Moreover,
\begin{equation}\label{eq:rank-one-inverse-norm}
    \|U_k^{-1}\|_{C([-a,a])\to C([-a,a])}
    =\frac{1+4k}{1+2k}.
\end{equation}
\end{proposition}

\begin{proof}
The case $k=0$ is immediate. Assume that $k>0$ and
put $b=2k$. For $\beta\geq0$, define
\[
    H_\beta h(x)
    =\int_0^1s^\beta h_{\mathrm e}(sx)\,ds,
    \qquad
    h_{\mathrm e}(x)=\frac{h(x)+h(-x)}2.
\]
These are bounded operators on $C([-a,a];\mathbb C)$.
Fubini's theorem gives
\[
\begin{aligned}
    H_0H_bh(x)=H_bH_0h(x)=\int_0^1\int_0^1
        s^b h_{\mathrm e}(stx)\,dt\,ds=\int_0^1h_{\mathrm e}(rx)
        \int_r^1s^{b-1}\,ds\,dr=\frac1b(H_0-H_b)h(x).
\end{aligned}
\]
Consequently,
\[
    (I+bH_0)(I-bH_b)
    =(I-bH_b)(I+bH_0)=I.
\]
Since $U_k=I+bH_0$, this proves
\eqref{eq:rank-one-inverse}. It also yields
\[
    \|U_k^{-1}h\|_\infty
    \leq
    \left(1+\frac{b}{b+1}\right)\|h\|_\infty
    =\frac{1+4k}{1+2k}\|h\|_\infty.
\]

To prove sharpness, consider the functional
\[
    h\longmapsto(U_k^{-1}h)(a)
    =h(a)-\frac{k}{a^{2k+1}}
        \int_{-a}^{a}|y|^{2k}h(y)\,dy.
\]
Its representing measure on $[-a,a]$ is
\[
    \delta_a-\frac{k}{a^{2k+1}}|y|^{2k}\,dy.
\]
The two positive measures on the right are mutually
singular. Hence the norm of this functional equals
\[
    1+\frac{k}{a^{2k+1}}
        \int_{-a}^{a}|y|^{2k}\,dy
    =1+\frac{2k}{2k+1}.
\]
This proves \eqref{eq:rank-one-inverse-norm}.

Finally, the fundamental theorem of calculus gives
\[
    Tf=U_k(f').
\]
Thus \eqref{eq:rank-one-estimate} follows from
\eqref{eq:rank-one-inverse-norm}.
Conversely, every continuous function on $[-a,a]$
is the derivative of a function in $C^1([-a,a])$.
Since $U_k$ is invertible, any constant valid in
\eqref{eq:rank-one-estimate} must therefore be at least
$\|U_k^{-1}\|$.
\end{proof}

\section{Harmonic estimates and vertical translates}\label{sec:pointwise}
\subsection{Pointwise estimates and logarithmic growth}

\begin{lemma}\label{lem:pointwise-logarithmic}
Let $u\in C^\infty(\mathbb R_+^{1+N};\mathbb C)$ satisfy
$\mathcal Lu=0$ and $\|u\|_{**}<\infty$. Then
\begin{equation}\label{eq:pointwise-Dunkl-gradient}
    |\partial_tu(t,\mathbf x)|
    +|\nabla_k u(t,\mathbf x)|
    \leq \frac{C}{t}\|u\|_{**},
    \qquad (t,\mathbf x)\in\mathbb R_+^{1+N}.
\end{equation}
Moreover,
\begin{equation}\label{eq:logarithmic-increment}
    |u(t,\mathbf x)-u(t,0)|
    \leq C\|u\|_{**}
    \log\left(1+\frac{\|\mathbf x\|}{t}\right).
\end{equation}
The constant $C$ depends only on the root system,
the multiplicity function, and the dimension.
\end{lemma}

\begin{proof}
Fix $(t,\mathbf x)\in\mathbb R_+^{1+N}$.
Since the Dunkl operators commute,
$v_j=T_ju$ is $\mathcal L$-harmonic for
$j=0,\ldots,N$.

We regard $R$ as a root system in $\mathbb R^{1+N}$,
with $G$ acting trivially on the first coordinate.
Write $\mathbf{z}=(t,\mathbf x)$, $\zeta=(s,\mathbf y)$, and
$d\mu(\zeta)=ds\,dw(\mathbf y)$, and denote Euclidean
balls in $\mathbb R^{1+N}$ by $B(z,r)$.
The corresponding homogeneous dimension is
$\mathbf N+1$.

Let $\widetilde\tau$ denote the Dunkl translation
in the extended space. For
$\mathbf z=(t,\mathbf x)$ and
$\boldsymbol\zeta=(s,\mathbf y)$, put
\[
    K_r(\mathbf z,\boldsymbol\zeta)
    =
    \frac{
        \widetilde\tau_{\mathbf z}
        \mathbf1_{B(0,r)}(-\boldsymbol\zeta)}
        {\mu(B(0,r))}.
\]
The positivity, normalization, and support properties
of the harmonic kernel
\cite[Proposition~3.1]{GR} give
\begin{equation}\label{eq:mean-kernel-properties}
    K_r(\mathbf z,\boldsymbol\zeta)\geq0,
    \qquad
    \int_{\mathbb R^{1+N}}
        K_r(\mathbf z,\boldsymbol\zeta)
        \,d\mu(\boldsymbol\zeta)=1,
    \qquad
    \operatorname{supp}K_r(\mathbf z,\cdot)
    \subseteq
    \bigcup_{g\in G}
        \overline{B((t,g\mathbf x),r)}.
\end{equation}

To estimate $K_r$, apply
\cite[Corollary~3.2]{ADH-JFAA} in the extended space
to the radial function
\[
    \phi(\boldsymbol\zeta)
    =
    \frac{\mathbf1_{B(0,1)}(\boldsymbol\zeta)}
         {\mu(B(0,1))}.
\]
Indeed, $\phi$ has arbitrarily rapid decay.
Since $d\mu=ds\,dw$ has homogeneous dimension
$\mathbf N+1$, we have
\[
    \mu(B(0,r))=r^{\mathbf N+1}\mu(B(0,1)),
\]
and hence
\[
    r^{-(\mathbf N+1)}
        \phi(\boldsymbol\zeta/r)
    =
    \frac{\mathbf1_{B(0,r)}(\boldsymbol\zeta)}
         {\mu(B(0,r))}.
\]
Taking the decay parameter $\kappa=0$ in that
corollary, we obtain
\begin{equation}\label{eq:mean-kernel-size}
    K_r(\mathbf z,\boldsymbol\zeta)
    \lesssim \frac1{\mu(B(\mathbf z,r))}
    \asymp \frac1{r\,w(B(\mathbf x,r))}.
\end{equation}
Here the volume comparison follows from
\[
    (t-r/2,t+r/2)\times B(\mathbf x,r/2)
    \subset B(\mathbf z,r)
    \subset (t-r,t+r)\times B(\mathbf x,r)
\]
and the doubling property of $dw$.

Now fix $t>0$ and set $r=t/4$.
Since $\mathcal L v_j=0$, the mean-value formula
\eqref{eq:GR-mean-value}, applied in the extended
space, gives
\[
    v_j(\mathbf z)
    =
    \int_{\mathbb R_+^{1+N}}
        K_r(\mathbf z,\boldsymbol\zeta)
        v_j(\boldsymbol\zeta)\,d\mu(\boldsymbol\zeta).
\]
By the Cauchy--Schwarz inequality and
\eqref{eq:mean-kernel-properties},
\[
    |v_j(\mathbf z)|^2
    \leq
    \int_{\mathbb R_+^{1+N}}
        K_r(\mathbf z,\boldsymbol\zeta)
        |v_j(\boldsymbol\zeta)|^2
        \,d\mu(\boldsymbol\zeta).
\]
Using \eqref{eq:mean-kernel-size}, the support
inclusion in \eqref{eq:mean-kernel-properties},
and \eqref{eq:volume-growth}, we conclude that
\begin{equation}\label{eq:local-mean-gradient}
\begin{aligned}
    |v_j(t,\mathbf x)|^2
    &\lesssim
    \frac1{t\,w(B(\mathbf x,t))}
    \sum_{g\in G}
    \int_{t/2}^{3t/2}
    \int_{B(g\mathbf x,t/2)}
        |v_j(s,\mathbf y)|^2
        \,dw(\mathbf y)\,ds.
\end{aligned}
\end{equation}

Since $s\geq t/2$ in \eqref{eq:local-mean-gradient},
we obtain
\[
\begin{aligned}
    |v_j(t,\mathbf x)|^2
    &\leq
    \frac{C}{t^2w(B(\mathbf x,t))}
    \sum_{g\in G}
    \int_0^{2t}\int_{B(g(\mathbf x),2t)}
    s\,|\widetilde\nabla_k u(s,\mathbf y)|^2
    \,dw(\mathbf y)\,ds\\&\leq
    \frac{C\|u\|_{**}^2}{t^2w(B(\mathbf x,t))}
    \sum_{g\in G}w(B(g(\mathbf x),2t))
    \leq \frac{C}{t^2}\|u\|_{**}^2,
\end{aligned}
\]
where the last inequality follows from the
$G$-invariance and doubling property of $dw$.
Summing over $j=0,\ldots,N$ proves
\eqref{eq:pointwise-Dunkl-gradient}.

To prove \eqref{eq:logarithmic-increment}, put $a=t+\|\mathbf x\|$.
Lemma~\ref{lem:increment-at-origin}, applied to
$u(a,\cdot)$, and \eqref{eq:pointwise-Dunkl-gradient}
give
\[
    |u(a,\mathbf x)-u(a,0)|
    \leq
    \|\mathbf x\|\sup_{\|\mathbf y\|\leq\|\mathbf x\|}
        |\nabla_k u(a,\mathbf y)|
    \leq C\|u\|_{**}\frac{\|\mathbf x\|}{a}.
\]
Integrating the estimate for $\partial_tu$ along
the two vertical segments, we obtain
\[
\begin{aligned}
    |u(t,\mathbf x)-u(t,0)|
    &\leq
    \int_t^a|\partial_su(s,\mathbf x)|\,ds
    +|u(a,\mathbf x)-u(a,0)|
    +\int_t^a|\partial_su(s,0)|\,ds\\
    &\leq
    C\|u\|_{**}
    \left(
        2\log\frac at+\frac{\|\mathbf x\|}{a}
    \right)\leq
    C\|u\|_{**}
    \log\left(1+\frac{\|\mathbf x\|}{t}\right),
\end{aligned}
\]
using $b/(1+b)\leq\log(1+b)$ for $b\geq0$.
\end{proof}

\begin{corollary}\label{cor:logarithmic-growth-of-traces}
Under the assumptions of
Lemma~\ref{lem:pointwise-logarithmic}, for every
$\varepsilon>0$ the function
$f_\varepsilon(\mathbf x)=u(\varepsilon,\mathbf x)$
satisfies
\[
    |f_\varepsilon(\mathbf x)|
    \leq |u(\varepsilon,0)|
    +C\|u\|_{**}
        \log\left(1+\frac{\|\mathbf x\|}{\varepsilon}\right).
\]
In particular, there exists a constant
$C_{u,\varepsilon}<\infty$ such that
\[
    |f_\varepsilon(\mathbf x)|
    \leq C_{u,\varepsilon}\log(2+\|\mathbf x\|),
    \qquad \mathbf x\in\mathbb R^N.
\]
\end{corollary}

\begin{proof}
This follows directly from
\eqref{eq:logarithmic-increment} with $t=\varepsilon$.
\end{proof}

\subsection{Poisson representation of vertical translates}

\begin{proposition}\label{prop:vertical-translates}
Let $u\in C^\infty(\mathbb R_+^{1+N};\mathbb C)$ satisfy
$\mathcal Lu=0$ and $\|u\|_{**}<\infty$. For
$\varepsilon>0$, set
\[
    f_\varepsilon(\mathbf x)=u(\varepsilon,\mathbf x).
\]
Then, for every $t>0$ and $\mathbf x\in\mathbb R^N$,
the integral
\[
    P_tf_\varepsilon(\mathbf x)
    =\int_{\mathbb R^N}
        p_t(\mathbf x,\mathbf y)f_\varepsilon(\mathbf y)
        \,dw(\mathbf y)
\]
is absolutely convergent and
\begin{equation}\label{eq:vertical-translate-representation}
    u(t+\varepsilon,\mathbf x)
    =P_tf_\varepsilon(\mathbf x).
\end{equation}
\end{proposition}

\begin{proof}
Fix $\varepsilon>0$ and write $f=f_\varepsilon$.
By Corollary~\ref{cor:logarithmic-growth-of-traces},
\begin{equation}\label{eq:trace-logarithmic-bound}
    |f(\mathbf y)|
    \leq C_{u,\varepsilon}\log(2+\|\mathbf y\|).
\end{equation}
The kernel estimates
\eqref{eq:Poisson-size}--\eqref{eq:Poisson-ordinary-derivatives}
therefore imply that $P_tf$ is absolutely convergent
and defines a smooth $\mathcal L$-harmonic function. Indeed, \eqref{eq:volume-estimate},
\eqref{eq:Poisson-size}, and
\eqref{eq:Poisson-ordinary-derivatives} imply that,
for every compact set $K\subset\mathbb R_+^{1+N}$,
\[
    \sup_{(t,\mathbf x)\in K}
    |\partial_t^m\partial_{\mathbf x}^{\beta}
        p_t(\mathbf x,\mathbf y)|
    \leq C_{K,m,\beta}\rho(\mathbf y).
\]
Since $f\in L^1(\rho\,dw)$, dominated convergence
justifies ordinary differentiation under the integral.
The same holds for Dunkl derivatives: their difference
quotients are averages of ordinary directional
derivatives along segments joining $\mathbf x$
to $\sigma_\alpha(\mathbf x)$.
Thus \eqref{eq:Poisson-differentiation} holds, and
the identity
$(\partial_t^2+\Delta_{k,\mathbf x})
p_t(\mathbf x,\mathbf y)=0$
implies that $P_tf$ is $\mathcal L$-harmonic.
Moreover, by \eqref{eq:Poisson-differentiation},
\begin{equation}\label{eq:derivative-under-Poisson-integral}
    T_jP_tf(\mathbf x)
    =\int_{\mathbb R^N}
        T_{j,\mathbf x}p_t(\mathbf x,\mathbf y)
        f(\mathbf y)\,dw(\mathbf y),
    \qquad j=1,\ldots,N.
\end{equation}

For $j=1,\ldots,N$, consider
\[
    v_j(t,\mathbf x)=T_ju(t+\varepsilon,\mathbf x).
\]
The commutation of the Dunkl operators implies that
$\mathcal Lv_j=0$. Moreover,
Lemma~\ref{lem:pointwise-logarithmic} gives
\[
    |v_j(t,\mathbf x)|
    \leq \frac{C}{t+\varepsilon}\|u\|_{**}
    \leq \frac{C}{\varepsilon}\|u\|_{**}.
\]
Thus $v_j$ is bounded and continuous on the closed
upper half-space, with boundary value $T_jf$.

For each $\varepsilon>0$, the function
$v_j(t,\mathbf x)$
is bounded, $\mathcal L$-harmonic, and continuous
on the closed upper half-space.
Since $v(0,\cdot)=T_jf_\varepsilon$,
Theorem~\ref{thm:bounded-Poisson-representation} yields
\begin{equation}\label{eq:Poisson-representation-of-derivatives}
    T_ju(t+\varepsilon,\mathbf x)
    =P_t(T_jf_{\varepsilon})(\mathbf x).
\end{equation}

We next verify that
\begin{equation}\label{eq:commutation-for-logarithmic-trace}
    P_t(T_jf_{\varepsilon})=T_jP_tf_{\varepsilon}.
\end{equation}
Choose a radial function $\chi\in C_c^\infty(\mathbb R^N)$
such that $0\leq\chi\leq1$, $\chi=1$ on $B(0,1)$,
and $\chi=0$ outside $B(0,2)$, and put
$\chi_L(\mathbf y)=\chi(\mathbf y/L)$ for $L\geq1$.
Since $\chi_L$ is $G$-invariant,
\[
    T_j(\chi_Lf_{\varepsilon})=\chi_LT_jf_{\varepsilon}+(\partial_j\chi_L)f_{\varepsilon}.
\]
Using
\[
    T_{j,\mathbf x}p_t(\mathbf x,\mathbf y)
    =-T_{j,\mathbf y}p_t(\mathbf x,\mathbf y)
\]
and the skew-symmetry of $T_j$, we obtain
\[
\begin{aligned}
    &\int_{\mathbb R^N}
        T_{j,\mathbf x}p_t(\mathbf x,\mathbf y)
        \chi_L(\mathbf y)f_{\varepsilon}(\mathbf y)\,dw(\mathbf y)=
    \int_{\mathbb R^N}
        p_t(\mathbf x,\mathbf y)
        \chi_L(\mathbf y)T_jf_{\varepsilon}(\mathbf y)\,dw(\mathbf y)\\
    &+
    \int_{\mathbb R^N}
        p_t(\mathbf x,\mathbf y)
        \partial_j\chi_L(\mathbf y)f_{\varepsilon}(\mathbf y)
        \,dw(\mathbf y).
\end{aligned}
\]
The left-hand side converges to $T_jP_tf_{\varepsilon}(\mathbf x)$
by \eqref{eq:derivative-under-Poisson-integral}.
Since $T_jf_{\varepsilon}$ is bounded, the first term on the right
converges to $P_t(T_jf_{\varepsilon})(\mathbf x)$.
Finally, the remaining term is bounded in modulus by
\[
    \frac{\|\partial_j\chi\|_\infty}{L}
    \int_{\mathbb R^N}
        p_t(\mathbf x,\mathbf y)|f_{\varepsilon}(\mathbf y)|
        \,dw(\mathbf y),
\]
and hence tends to zero. This proves
\eqref{eq:commutation-for-logarithmic-trace}.

Set
\[
    \psi(t,\mathbf x)
    =u(t+\varepsilon,\mathbf x)-P_tf_{\varepsilon}(\mathbf x).
\]
Equations
\eqref{eq:Poisson-representation-of-derivatives}
and \eqref{eq:commutation-for-logarithmic-trace} imply
\[
    T_j\psi(t,\cdot)=0,
    \qquad j=1,\ldots,N.
\]

Since $T_j\psi(t,\cdot)=0$ for every $j$,
Lemma~\ref{lem:increment-at-origin} implies that
$\psi(t,\mathbf x)=\psi(t,0)$.
Both terms defining $\psi$ are $\mathcal L$-harmonic.
Hence $\mathcal L\psi=0$ and, since $\psi$ is independent
of $\mathbf x$, we obtain $\partial_t^2\psi(t,0)=0$.
Consequently,
\begin{equation}\label{eq:affine-remainder}
    \psi(t,\mathbf x)=at+b
\end{equation}
for some $a,b\in\mathbb C$.

To eliminate the linear term, recall that
\[
    p_t(0,\mathbf y)
    =c\,\frac{t}
        {(t^2+\|\mathbf y\|^2)^{(\mathbf N+1)/2}},
    \qquad
    \mathbf N=N+\sum_{\alpha\in R}k(\alpha).
\]
By scaling and \eqref{eq:trace-logarithmic-bound},
\[
\begin{aligned}
    |P_tf_{\varepsilon}(0)|
    &\leq C_{u,\varepsilon}
        \int_{\mathbb R^N}
        p_1(0,\mathbf z)\log(2+t\|\mathbf z\|)
        \,dw(\mathbf z)\leq C'_{u,\varepsilon}\log(2+t),
\end{aligned}
\]
where we used
\[
    \int_{\mathbb R^N}
        p_1(0,\mathbf z)\log(1+\|\mathbf z\|)
        \,dw(\mathbf z)<\infty.
\]
Similarly, integrating
\eqref{eq:pointwise-Dunkl-gradient} at $\mathbf x=0$
gives
\[
    |u(t+\varepsilon,0)|
    \leq C''_{u,\varepsilon}\log(2+t),
    \qquad t\geq1.
\]
Thus \eqref{eq:affine-remainder} forces $a=0$.

Finally, by \eqref{eq:Poisson-mass},
\[
    P_tf_{\varepsilon}(0)-f_{\varepsilon}(0)
    =
    \int_{\mathbb R^N}p_t(0,\mathbf y)
        \bigl(f_{\varepsilon}(\mathbf y)-f_{\varepsilon}(0)\bigr)\,dw(\mathbf y).
\]
Given $\eta>0$, continuity provides $\delta>0$ such
that $|f_{\varepsilon}(\mathbf y)-f_{\varepsilon}(0)|<\eta$ for $\|\mathbf y\|<\delta$.
Using \eqref{eq:Poisson-at-origin}, we obtain
\[
    |P_tf_{\varepsilon}(0)-f_{\varepsilon}(0)|
    \leq \eta+
    c_{R,k}t\int_{\|\mathbf y\|\geq\delta}
        \frac{|f_{\varepsilon}(\mathbf y)-f_{\varepsilon}(0)|}
             {\|\mathbf y\|^{\mathbf N+1}}\,dw(\mathbf y).
\]
The last integral is finite by
\eqref{eq:trace-logarithmic-bound}; in polar coordinates
its tail is bounded by a constant multiple of
$\int_1^\infty(1+\log s)s^{-2}\,ds$.
Letting first $t\to0^+$ and then $\eta\to0^+$ gives
$P_tf_{\varepsilon}(0)\to f_{\varepsilon}(0)$.
Since $f_{\varepsilon}=u(\varepsilon,\cdot)$, continuity of $u$ yields
\[
    b=\lim_{t\to0^+}
        \bigl(u(t+\varepsilon,0)-P_tf_{\varepsilon}(0)\bigr)
      =u(\varepsilon,0)-f_{\varepsilon}(0)=0.
\]
This proves \eqref{eq:vertical-translate-representation}.
\end{proof}

\section{A Carleson measure criterion for BMO}\label{sec:carleson}

\begin{theorem}\label{thm:Carleson-implies-BMO}
Let $f\in L^1_{\mathrm{loc}}(dw)$ satisfy
\begin{equation}\label{eq:weighted-integrability-BMO}
    \int_{\mathbb R^N}
    \frac{|f(\mathbf x)|}
         {(1+\|\mathbf x\|)^{\mathbf N+1}}
    \,dw(\mathbf x)<\infty.
\end{equation}
If
\[
    d\nu_f(t,\mathbf x)
    =|Q_tf(\mathbf x)|^2\,dw(\mathbf x)\,\frac{dt}{t}
\]
is a Carleson measure, then
$f\in\mathrm{BMO}(dw)$ and
\begin{equation}\label{eq:Carleson-to-BMO-bound}
    \|f\|_{\mathrm{BMO}(dw)}
    \leq C\|\nu_f\|_{\mathcal C}^{1/2}.
\end{equation}
\end{theorem}

\begin{proof}
Write $q_t(\mathbf x,\mathbf y)
=-t\partial_tp_t(\mathbf x,\mathbf y)$. Then 
\[Q_tf(\mathbf{x})=\int_{\mathbb{R}^N}q_t(\mathbf{x},\mathbf{y})f(\mathbf{y})\,dw(\mathbf{y}),\]
and put
\[
    \rho(\mathbf x)=p_1(0,\mathbf x)
    \asymp(1+\|\mathbf x\|)^{-\mathbf N-1}.
\]
Thus \eqref{eq:weighted-integrability-BMO} is equivalent
to $\int |f|\rho\,dw<\infty$.
The Poisson kernel estimates ensure that $P_tf$ and
$Q_tf$ are well defined by absolutely convergent
integrals.

We first establish the reproducing identity
\begin{equation}\label{eq:BMO-reproducing}
    \int_{\mathbb R^N}f(\mathbf x)g(\mathbf x)
        \,dw(\mathbf x)
    =
    4\int_0^\infty\int_{\mathbb R^N}
        Q_tf(\mathbf x)Q_tg(\mathbf x)
        \,dw(\mathbf x)\,\frac{dt}{t}
\end{equation}
for every bounded, compactly supported function $g$
satisfying $\int g\,dw=0$.

The estimates \eqref{eq:Poisson-size},
\eqref{eq:Poisson-ordinary-derivatives}, and
\eqref{eq:definition-Q}, together with the cancellation
of $g$, imply
\begin{equation}\label{eq:test-function-envelope}
    |g(\mathbf x)|
    +\sup_{t>0}|P_tg(\mathbf x)|
    +\sup_{t>0}|Q_tg(\mathbf x)|
    \leq C_g\rho(\mathbf x).
\end{equation}
Indeed, choose $R\geq1$ such that
$\operatorname{supp}g\subset B(0,R)$.
By positivity and \eqref{eq:Poisson-mass},
$\|P_tg\|_\infty\leq\|g\|_\infty$.
Moreover, \eqref{eq:definition-Q} and
\eqref{eq:Poisson-ordinary-derivatives} give
$|q_t(\mathbf x,\mathbf y)|\lesssim
p_t(\mathbf x,\mathbf y)$, and hence
$\|Q_tg\|_\infty\lesssim\|g\|_\infty$.
Since $\rho$ is bounded below on $B(0,2R)$,
these bounds prove \eqref{eq:test-function-envelope}
there.

Now suppose that $\|\mathbf x\|>2R$.
For $\|\mathbf y\|\leq R$ we have
\[
    \tfrac12\|\mathbf x\|
    \leq d(\mathbf x,\mathbf y)
    \leq\tfrac32\|\mathbf x\|.
\]
By \eqref{eq:volume-estimate},
\[
    w(B(\mathbf x,t+d(\mathbf x,\mathbf y)))
    \gtrsim(t+\|\mathbf x\|)^{\mathbf N}.
\]
Consequently, \eqref{eq:Poisson-size} yields
\begin{equation}\label{eq:test-kernel-size}
    |p_t(\mathbf x,\mathbf y)|
    +|q_t(\mathbf x,\mathbf y)|
    \lesssim
    \frac{t}{(t+\|\mathbf x\|)^{\mathbf N+1}}.
\end{equation}
For $0<t\leq R$, integration against $|g|$ gives
\[
    |P_tg(\mathbf x)|+|Q_tg(\mathbf x)|
    \lesssim
    R\|g\|_{L^1(dw)}\|\mathbf x\|^{-\mathbf N-1}.
\]

For $t>R$, we use cancellation. By
\eqref{eq:Poisson-ordinary-derivatives} and
$q_t=-t\partial_tp_t$,
\[
    |\nabla_{\mathbf y}p_t(\mathbf x,\mathbf y)|
    +|\nabla_{\mathbf y}q_t(\mathbf x,\mathbf y)|
    \lesssim t^{-1}p_t(\mathbf x,\mathbf y)
    \lesssim(t+\|\mathbf x\|)^{-\mathbf N-1},
    \qquad \|\mathbf y\|\leq R.
\]
Here $\nabla_{\mathbf y}$ denotes the ordinary gradient.
The mean value theorem therefore gives, for
$K_t=p_t$ or $K_t=q_t$,
\[
    |K_t(\mathbf x,\mathbf y)-K_t(\mathbf x,0)|
    \lesssim
    \|\mathbf y\|(t+\|\mathbf x\|)^{-\mathbf N-1}.
\]
Since $\int g\,dw=0$, it follows that
\[
\begin{aligned}
    |P_tg(\mathbf x)|+|Q_tg(\mathbf x)|
    &\lesssim
    (t+\|\mathbf x\|)^{-\mathbf N-1}
    \int_{\mathbb R^N}\|\mathbf y\||g(\mathbf y)|
        \,dw(\mathbf y)\lesssim
    R\|g\|_{L^1(dw)}\|\mathbf x\|^{-\mathbf N-1}.
\end{aligned}
\]
Together with
$\rho(\mathbf x)\asymp(1+\|\mathbf x\|)^{-\mathbf N-1}$,
this proves \eqref{eq:test-function-envelope}.

For $0<a<b<\infty$, the spectral theorem gives
\begin{equation}\label{eq:truncated-reproducing-BMO}
    4\int_a^b Q_t^2\,\frac{dt}{t}
    =P_{2a}-P_{2b}+Q_{2a}-Q_{2b}
\end{equation}
on $L^2(dw)$.
To pair this identity with $f$, observe that
$|q_t(\mathbf x,\mathbf y)|\leq Cp_t(\mathbf x,\mathbf y)$.
By \eqref{eq:test-function-envelope}, symmetry,
and the semigroup property,
\[
\begin{aligned}
    &\int_{\mathbb R^N}
        |q_t(\mathbf x,\mathbf y)|\,|Q_tg(\mathbf x)|
        \,dw(\mathbf x)\leq
        C_g\int_{\mathbb R^N}
        p_t(\mathbf x,\mathbf y)\rho(\mathbf x)
        \,dw(\mathbf x)=C_g p_{t+1}(0,\mathbf y)
        \leq C_{g,a,b}\rho(\mathbf y)
\end{aligned}
\]
for $a\leq t\leq b$. Hence all relevant integrals on $[a,b]$ are absolutely
convergent. Fubini's theorem and
\eqref{eq:truncated-reproducing-BMO} therefore yield
\begin{equation}\label{eq:paired-truncated-reproducing}
\begin{aligned}
    &4\int_a^b\int_{\mathbb R^N}
        Q_tf(\mathbf x)Q_tg(\mathbf x)
        \,dw(\mathbf x)\,\frac{dt}{t}\\
    &\qquad=
    \int_{\mathbb R^N}
        f(\mathbf x)
        \bigl(P_{2a}g-P_{2b}g+Q_{2a}g-Q_{2b}g\bigr)
        (\mathbf x)\,dw(\mathbf x).
\end{aligned}
\end{equation}

Put $L=\sqrt{-{\Delta_k}}$. Since $g\in L^1(dw)\cap L^2(dw)$,
the spectral theorem and dominated convergence give
\[
    P_{2a}g=e^{-2aL}g\longrightarrow g,
    \qquad
    Q_{2a}g=2aLe^{-2aL}g\longrightarrow0
    \quad\text{in }L^2(dw)\quad(a\to0^+).
\]
Indeed, for every $\lambda\geq0$,
$e^{-2a\lambda}\to1$ and
$2a\lambda e^{-2a\lambda}\to0$, while
\[
    |e^{-2a\lambda}|\leq1,
    \qquad
    |2a\lambda e^{-2a\lambda}|\leq e^{-1}.
\]

For the limits at infinity,~\eqref{eq:volume-estimate},~\eqref{eq:Poisson-size}, and~\eqref{eq:Poisson-ordinary-derivatives} give
\[
    |q_s(\mathbf x,\mathbf y)|
    \lesssim p_s(\mathbf x,\mathbf y)
    \lesssim s^{-\mathbf N}.
\]
By symmetry and normalization of $p_s$, these estimates
imply, for $A_s=P_s$ or $A_s=Q_s$,
\[
    \|A_sg\|_{L^1(dw)}\lesssim\|g\|_{L^1(dw)},
    \qquad
    \|A_sg\|_{L^\infty(dw)}
    \lesssim s^{-\mathbf N}\|g\|_{L^1(dw)}.
\]
Consequently,
\[
    \|A_sg\|_{L^2(dw)}^2
    \leq
    \|A_sg\|_{L^\infty(dw)}
    \|A_sg\|_{L^1(dw)}
    \lesssim s^{-\mathbf N}\|g\|_{L^1(dw)}^2
    \longrightarrow0
    \qquad(s\to\infty).
\]
Taking $s=2b$, we obtain
\[
    P_{2b}g\longrightarrow0,
    \qquad
    Q_{2b}g\longrightarrow0
    \quad\text{in }L^2(dw)\quad(b\to\infty).
\]
These limits remain valid after integration against $f$.
Indeed, approximate $f$ in $L^1(\rho\,dw)$ by bounded,
compactly supported functions. For the approximants,
the assertion follows from $L^2$ convergence; the
errors are uniformly controlled by
\eqref{eq:test-function-envelope}.

It remains to prove absolute convergence of the integral
over all $t>0$. We use the atomic decomposition of tent
spaces on the doubling space
$(\mathbb R^N,\|\cdot-\cdot\|,dw)$.

For a measurable function $H$ on $\mathbb R_+^{1+N}$, set
\[
    \mathcal A H(\mathbf x)
    =
    \left(
        \int_0^\infty
        \int_{\|\mathbf x-\mathbf y\|<t}
            |H(t,\mathbf y)|^2
            \frac{dw(\mathbf y)\,dt}
                 {t\,w(B(\mathbf x,t))}
    \right)^{1/2},
    \qquad
    \|H\|_{T^1_2}=\|\mathcal A H\|_{L^1(dw)}.
\]
A $T^1_2$-atom associated with a ball $B$ is a function
$A$ supported in $\widehat B$ and satisfying
\[
    \int_0^\infty\int_{\mathbb R^N}
        |A(t,\mathbf x)|^2\,dw(\mathbf x)\,\frac{dt}{t}
    \leq w(B)^{-1}.
\]
By \cite[Theorem~1.1]{Russ}, every $H\in T^1_2$
admits a decomposition
\[
    H=\sum_{\ell=1}^\infty\lambda_\ell A_\ell,
    \qquad
    \sum_{\ell=1}^\infty|\lambda_\ell|
    \leq C\|H\|_{T^1_2},
\]
with convergence in $T^1_2$. A subsequence of the
partial sums converges almost everywhere with respect
to $dw(\mathbf x)\,dt/t$.

We apply this result to $H(t,\mathbf x)=Q_tg(\mathbf x)$.
By the definition of $Sg$ and
\eqref{eq:H1-square-characterization},
\[
    \|H\|_{T^1_2}
    =\|Sg\|_{L^1(dw)}
    \lesssim\|g\|_{H^1_{\Delta_k}}.
\]
Thus
\begin{equation}\label{eq:Qtg-tent-decomposition}
    Q_tg(\mathbf x)
    =\sum_{\ell=1}^\infty\lambda_\ell A_\ell(t,\mathbf x),
    \qquad
    \sum_{\ell=1}^\infty|\lambda_\ell|
    \lesssim\|Sg\|_{L^1(dw)}.
\end{equation}

Let $A_\ell$ be associated with the ball $B_\ell$.
Since
\[
    d\nu_f(t,\mathbf x)
    =|Q_tf(\mathbf x)|^2\,dw(\mathbf x)\,\frac{dt}{t},
\]
the Carleson condition and the equivalence of the
tent and box definitions give
\[
    \int_{\widehat B_\ell}|Q_tf(\mathbf x)|^2
        \,dw(\mathbf x)\,\frac{dt}{t}
    =\nu_f(\widehat B_\ell)
    \lesssim\|\nu_f\|_{\mathcal C}\,w(B_\ell).
\]
Using the support and normalization of $A_\ell$,
the Cauchy--Schwarz inequality therefore yields
\[
\begin{aligned}
    &\int_0^\infty\int_{\mathbb R^N}
        |Q_tf(\mathbf x)|\,|A_\ell(t,\mathbf x)|
        \,dw(\mathbf x)\,\frac{dt}{t}\\
    &\quad\leq
    \left(
        \int_{\widehat B_\ell}|Q_tf(\mathbf x)|^2
            \,dw(\mathbf x)\,\frac{dt}{t}
    \right)^{1/2}
    \left(
        \int_{\widehat B_\ell}|A_\ell(t,\mathbf x)|^2
            \,dw(\mathbf x)\,\frac{dt}{t}
    \right)^{1/2}\\
    &\quad\lesssim
        \|\nu_f\|_{\mathcal C}^{1/2}
        w(B_\ell)^{1/2}w(B_\ell)^{-1/2}
    =\|\nu_f\|_{\mathcal C}^{1/2}.
\end{aligned}
\]
Applying this estimate to finite partial sums in
\eqref{eq:Qtg-tent-decomposition} and then using
Fatou's lemma along an almost everywhere convergent
subsequence, we obtain
\begin{equation}\label{eq:Carleson-Hardy-pairing}
\begin{aligned}
    &\int_0^\infty\int_{\mathbb R^N}
        |Q_tf(\mathbf x)|\,|Q_tg(\mathbf x)|
        \,dw(\mathbf x)\,\frac{dt}{t}\lesssim
        \|\nu_f\|_{\mathcal C}^{1/2}
        \sum_{\ell=1}^\infty|\lambda_\ell|\\
    &\qquad\lesssim
        \|\nu_f\|_{\mathcal C}^{1/2}
        \|Sg\|_{L^1(dw)}
    \lesssim
        \|\nu_f\|_{\mathcal C}^{1/2}
        \|g\|_{H^1_{\Delta_k}}.
\end{aligned}
\end{equation}

In particular, the integral in
\eqref{eq:BMO-reproducing} is absolutely convergent.
Letting $a\to0^+$ and $b\to\infty$ in
\eqref{eq:paired-truncated-reproducing} proves
\eqref{eq:BMO-reproducing}. Combining it with
\eqref{eq:Carleson-Hardy-pairing}, we obtain
\begin{equation}\label{eq:bounded-pairing-BMO}
    \left|\int_{\mathbb R^N}fg\,dw\right|
    \leq C\|\nu_f\|_{\mathcal C}^{1/2}
        \|g\|_{H^1_{\Delta_k}}.
\end{equation}

Finally, fix a Euclidean ball $B$ and a measurable
function $h$ on $B$ with $|h|\leq1$, and set
\[
    g=(h-h_B)\chi_B.
\]
Then $g/(2w(B))$ is a $(1,\infty)$-atom, and hence
$\|g\|_{H^1_{\Delta_k}}\leq Cw(B)$.
Applying \eqref{eq:bounded-pairing-BMO} gives
\[
    \left|\int_B(f-f_B)h\,dw\right|
    =\left|\int_{\mathbb R^N}fg\,dw\right|
    \leq C\|\nu_f\|_{\mathcal C}^{1/2}w(B).
\]
Taking the supremum over $|h|\leq1$ and then over $B$
proves \eqref{eq:Carleson-to-BMO-bound}.
\end{proof}

\begin{lemma}\label{lem:uniform-BMO-traces}
Let $u\in C^\infty(\mathbb R_+^{1+N};\mathbb C)$ satisfy
$\mathcal Lu=0$ and $\|u\|_{**}<\infty$. For
$\varepsilon>0$, set
\[
    u_\varepsilon(t,\mathbf x)=u(t+\varepsilon,\mathbf x),
    \qquad
    f_\varepsilon(\mathbf x)=u(\varepsilon,\mathbf x).
\]
Then
\begin{equation}\label{eq:uniform-shift-Carleson}
    \|u_\varepsilon\|_{**}\leq C\|u\|_{**}
\end{equation}
and
\begin{equation}\label{eq:uniform-BMO-traces}
    \|f_\varepsilon\|_{\mathrm{BMO}(dw)}
    \leq C\|u\|_{**},
\end{equation}
with $C$ independent of $u$ and $\varepsilon$.
\end{lemma}

\begin{proof}
Put $M=\|u\|_{**}$.
Fix a ball $B=B(\mathbf x_0,r)$ and write
\[
    I_\varepsilon(B)
    =
    \frac1{w(B)}
    \int_0^r\int_B
    t\,|\widetilde\nabla_k u(t+\varepsilon,\mathbf x)|^2
    \,dw(\mathbf x)\,dt.
\]
If $r\geq\varepsilon$, the change of variables
$s=t+\varepsilon$ gives
\[
\begin{aligned}
    I_\varepsilon(B)
    &=
    \frac1{w(B)}
    \int_\varepsilon^{r+\varepsilon}\int_B
    (s-\varepsilon)
    |\widetilde\nabla_k u(s,\mathbf x)|^2
    \,dw(\mathbf x)\,ds\\
    &\leq
    \frac1{w(B)}
    \int_0^{2r}\int_{2B}
    s\,|\widetilde\nabla_k u(s,\mathbf x)|^2
    \,dw(\mathbf x)\,ds\leq \frac{w(2B)}{w(B)}M^2
    \leq CM^2.
\end{aligned}
\]
If $r<\varepsilon$, the pointwise estimate
\eqref{eq:pointwise-Dunkl-gradient} yields
\[
\begin{aligned}
    I_\varepsilon(B)
    &\leq CM^2\int_0^r\frac{t}{(t+\varepsilon)^2}\,dt\leq CM^2\frac{r^2}{2\varepsilon^2}
    \leq CM^2.
\end{aligned}
\]
Taking the supremum over $B$ proves
\eqref{eq:uniform-shift-Carleson}.

By Corollary~\ref{cor:logarithmic-growth-of-traces},
$f_\varepsilon$ has at most logarithmic growth.
In particular,
\[
    \int_{\mathbb R^N}
    \frac{|f_\varepsilon(\mathbf x)|}
         {(1+\|\mathbf x\|)^{\mathbf N+1}}
    \,dw(\mathbf x)<\infty.
\]
Moreover, Proposition~\ref{prop:vertical-translates}
and the identity $Q_t=-t\partial_tP_t$ give
\[
    Q_tf_\varepsilon(\mathbf x)
    =-t\partial_tu_\varepsilon(t,\mathbf x).
\]
Consequently, the measure
\[
    d\nu_\varepsilon(t,\mathbf x)
    =|Q_tf_\varepsilon(\mathbf x)|^2
        \,dw(\mathbf x)\,\frac{dt}{t}
\]
satisfies
\[
    \|\nu_\varepsilon\|_{\mathcal C}
    \leq \|u_\varepsilon\|_{**}^2
    \leq CM^2.
\]
Applying Theorem~\ref{thm:Carleson-implies-BMO},
we conclude that
\[
    \|f_\varepsilon\|_{\mathrm{BMO}(dw)}
    \leq C\|\nu_\varepsilon\|_{\mathcal C}^{1/2}
    \leq CM,
\]
as required.
\end{proof}

\section{Proof of the main theorem}
\label{sec:proof-main-theorem}

We first record a consequence of the atomic
characterization of $H^1_{\Delta_k}$.

\begin{lemma}\label{lem:decaying-Hardy-function}
Let $h$ be a measurable function on $\mathbb R^N$
such that
\[
    |h(\mathbf y)|
    \leq C_h(1+\|\mathbf y\|)^{-\mathbf N-1},
    \qquad
    \int_{\mathbb R^N}h\,dw=0.
\]
Then $h\in H^1_{\Delta_k}$ and
$\|h\|_{H^1_{\Delta_k}}\lesssim C_h$.
\end{lemma}

\begin{proof}
Put $B_j=B(0,2^j)$ for $j\geq0$, and let
\[
    E_0=B_0,
    \qquad E_j=B_j\setminus B_{j-1}\quad(j\geq1),
    \qquad
    \varphi_j=\frac{\mathbf1_{B_j}}{w(B_j)}.
\]
Set
\[
    m_j=\int_{E_j}h\,dw,
    \qquad
    h_j=h\mathbf1_{E_j}-m_j\varphi_j.
\]
Since $w(B_j)=2^{j\mathbf N}w(B_0)$, the decay
assumption gives
\[
    |m_j|\lesssim C_h2^{-j},
    \qquad
    \|h_j\|_\infty
    \lesssim C_h2^{-j}w(B_j)^{-1}.
\]
Moreover, $h_j$ is supported in $B_j$ and has
integral zero. Thus each $h_j$ is a scalar multiple
of a $(1,\infty)$-atom, with coefficient bounded
by $CC_h2^{-j}$.

Define
\[
    b_j=\sum_{\ell=j+1}^\infty m_\ell.
\]
Then $|b_j|\lesssim C_h2^{-j}$, and the cancellation
$\sum_{j\geq0}m_j=0$ gives
\[
    h
    =
    \sum_{j=0}^\infty h_j
    +
    \sum_{j=0}^\infty
        b_j(\varphi_{j+1}-\varphi_j).
\]
Both series converge in $L^1(dw)$.
The function $\varphi_{j+1}-\varphi_j$ has integral
zero, is supported in $B_{j+1}$, and satisfies
\[
    \|\varphi_{j+1}-\varphi_j\|_\infty
    \lesssim w(B_{j+1})^{-1}.
\]
Consequently, the displayed representation is an
atomic decomposition with sum of absolute
coefficients bounded by
$CC_h\sum_{j\geq0}2^{-j}$.
The atomic characterization of $H^1_{\Delta_k}$
therefore proves the assertion.
\end{proof}

\begin{proof}[Proof of Theorem~\ref{thm:main}]
Put $M=\|u\|_{**}$ and, for $n\geq1$, set
\[
    f_n(\mathbf x)=u(n^{-1},\mathbf x).
\]
By Lemma~\ref{lem:uniform-BMO-traces},
\[
    \sup_{n\geq1}\|f_n\|_{\mathrm{BMO}(dw)}
    \leq CM.
\]
Hence, by \eqref{eq:H1-BMO-duality}, the associated
functionals $\Phi_{f_n}$ form a bounded sequence
in $(H^1_{\Delta_k})^*$.
The Banach--Alaoglu theorem and the separability
of $H^1_{\Delta_k}$ yield a subsequence, still denoted
by $(f_n)$, and $f\in\mathrm{BMO}(dw)$ such that
\begin{equation}\label{eq:boundary-weak-star-limit}
    \langle\Phi_{f_n},h\rangle
    \longrightarrow
    \langle\Phi_f,h\rangle
    \qquad\text{for every }h\in H^1_{\Delta_k}.
\end{equation}
Moreover,
\begin{equation}\label{eq:boundary-BMO-bound}
    \|f\|_{\mathrm{BMO}(dw)}\leq CM.
\end{equation}
Choose any representative of the BMO class $f$.
By \eqref{eq:BMO-weighted-integrability},
$f\in L^1(\rho\,dw)$, so its Poisson integral
is absolutely convergent.

Fix $t>0$ and $\mathbf x\in\mathbb R^N$, and define
\[
    h_{t,\mathbf x}(\mathbf y)
    =p_t(\mathbf x,\mathbf y)-p_1(0,\mathbf y).
\]
By \eqref{eq:Poisson-mass},
\[
    \int_{\mathbb R^N}h_{t,\mathbf x}(\mathbf y)
        \,dw(\mathbf y)=0.
\]
Furthermore, \eqref{eq:volume-estimate} and
\eqref{eq:Poisson-size} imply
\[
    p_t(\mathbf x,\mathbf y)
    \leq C_{t,\mathbf x}
        (1+\|\mathbf y\|)^{-\mathbf N-1}.
\]
Indeed, for $\|\mathbf y\|$ sufficiently large,
$d(\mathbf x,\mathbf y)\asymp\|\mathbf y\|$
and
\[
    w(B(\mathbf x,t+d(\mathbf x,\mathbf y)))
    \asymp\|\mathbf y\|^{\mathbf N};
\]
on bounded sets the assertion follows from
continuity of the kernel.
Together with \eqref{eq:Poisson-at-origin}, this gives
\[
    |h_{t,\mathbf x}(\mathbf y)|
    \leq C_{t,\mathbf x}
        (1+\|\mathbf y\|)^{-\mathbf N-1}.
\]
Lemma~\ref{lem:decaying-Hardy-function} therefore
shows that $h_{t,\mathbf x}\in H^1_{\Delta_k}$.

Since $f_n,f\in L^1(\rho\,dw)$, all the products
$f_nh_{t,\mathbf x}$ and $fh_{t,\mathbf x}$ are
integrable. Thus
\eqref{eq:absolutely-integrable-pairing} and
\eqref{eq:boundary-weak-star-limit} imply
\begin{equation}\label{eq:Poisson-difference-limit}
\begin{aligned}
    P_tf_n(\mathbf x)-P_1f_n(0)
    &=\langle\Phi_{f_n},h_{t,\mathbf x}\rangle\longrightarrow
      \langle\Phi_f,h_{t,\mathbf x}\rangle=P_tf(\mathbf x)-P_1f(0).
\end{aligned}
\end{equation}
On the other hand, Proposition~\ref{prop:vertical-translates} gives
\[
    P_tf_n(\mathbf x)=u(t+n^{-1},\mathbf x),
    \qquad
    P_1f_n(0)=u(1+n^{-1},0).
\]
Passing to the limit in
\eqref{eq:Poisson-difference-limit} and using
continuity of $u$, we obtain
\[
    u(t,\mathbf x)-u(1,0)
    =P_tf(\mathbf x)-P_1f(0).
\]

Finally, set
\[
    c=u(1,0)-P_1f(0),
    \qquad
    F=f+c.
\]
Since $P_t1=1$, we have
\[
    u(t,\mathbf x)=P_tF(\mathbf x)
    =\int_{\mathbb R^N}
        p_t(\mathbf x,\mathbf y)F(\mathbf y)
        \,dw(\mathbf y).
\]
The integral is absolutely convergent, and
\eqref{eq:boundary-BMO-bound} yields
\[
    \|F\|_{\mathrm{BMO}(dw)}
    =\|f\|_{\mathrm{BMO}(dw)}
    \leq C\|u\|_{**}.
\]
This proves the theorem.
\end{proof}

\end{document}